\documentclass[12pt]{amsart}

\numberwithin{equation}{section}

\usepackage{amsfonts, amsmath, amssymb, amsgen, amsthm, amscd,
latexsym}
\usepackage{color}
\usepackage[all]{xy}

\newtheorem{thm}{Theorem}[section]

\newtheorem{cor}[thm]{Corollary}
\newtheorem{lem}[thm]{Lemma}

\def\Diff{\mathop{\rm Diff}\nolimits}

\def\GL{\mathop{\rm GL}\nolimits}

\def\O{\mathop{\rm O}\nolimits}

\def\Id{\mathop{\rm Id}\nolimits}

\def\det{\mathop{\rm det}\nolimits}

\def\Tr{\mathop{\rm Tr}\nolimits}

\newcommand{\ra}{\rightarrow}
\newcommand{\fl}{\forall}
\newcommand{\wt}{\widetilde}
\newcommand{\wg}{\wedge}

\newcommand{\s}{\sigma}

\newcommand{\D}{\Delta}

\newcommand{\Zb}{\mathbb{Z}}

\newcommand{\ot}{\otimes}

\newcommand{\Hc}{\mathcal{H}}
\newcommand{\g}{\gamma}
\newcommand{\vp}{\varphi}
\newcommand{\ve}{\varepsilon}

\newcommand{\Cb}{\mathbb{C}}

\newcommand{\FA}{\mathfrak{A}}

\newcommand{\Fa}{\mathfrak{a}}
\newcommand{\Fg}{\mathfrak{g}}
\newcommand{\Fh}{\mathfrak{h}}

\newcommand{\Fl}{\mathfrak{l}}

\newcommand{\p}{\partial}

\def\tb{{\bf t}}

\def\Gb{{\bf G}}

\def\0b{{\bf 0}}

\def\Cb{{\mathbb C}}

\def\Nb{{\mathbb N}}
\def\Rb{{\mathbb R}}

\def\Zb{{\mathbb Z}}

\def\Ac{{\mathcal A}}
\def\Bc{{\mathcal B}}
\def\Cc{{\mathcal C}}
\def\Dc{{\mathcal D}}

\def\Fc{{\mathcal F}}
\def\Gc{{\mathcal G}}
\def\Hc{{\mathcal H}}
\def\Ic{{\mathcal I}}

\def\Qc{{\mathcal Q}}

\def\Uc{{\mathcal U}}

\def\Xc{{\mathcal X}}

\def\a{\alpha}

\def\d{\delta}

\def\lb{\lambda}
\def\g{\gamma}
\def\om{\omega}
\def\s{\sigma}
\def\t{\theta}
\def\ve{\varepsilon}
\def\vp{\varphi}

\def\vp{\varphi}

\def\D{\Delta}
\def\G{\Gamma}

\def\Om{\Omega}

\def\fl{\forall}
\def\ify{\infty}

\def\nb{\nabla}

\def\ot{\otimes}
\def\part{\partial}

\def\ts{\times}
\def\wdg{\wedge}

\def\ra{\rightarrow}

\def\text{\hbox}

\def\Diff{\mathop{\rm Diff}\nolimits}

\def\fl{\forall}
\def\ify{\infty}

\def\nb{\nabla}
\def\ot{\otimes}

\def\ra{\rightarrow}

\def\p{\partial}

\def\wt{\widetilde}
\def\td{\tilde}

\def\0D{\Delta^{(0)}}
\def\1D{\Delta^{(1)}}

\def\dbo{{\bf d}}

\def\tb{{\bf t}}

\def\Gb{{\bf G}}

\def\0b{{\bf 0}}
 
\def\build#1_#2^#3{\mathrel{
\mathop{\kern 0pt#1}\limits_{#2}^{#3}}}

\def\one{{\bf 1}}

\def\one{{\bf 1}}

\def\0D{\Delta^{(0)}}
\def\1D{\Delta^{(1)}}

\title[Equivariant Chern classes]{Equivariant Chern classes in Hopf cyclic
cohomology}

\begin{document}
\author{Henri Moscovici}
 \address{Department of mathematics,  
The Ohio State University, 
Columbus, OH 43210, USA}
\email{henri@math.osu.edu}
 
\thanks{Supported in part by the NSF
award DMS-1300548 and by the CNCS-UEFISCDI project PN-II-ID-PCE-2012-4-0201} 


 
\begin{abstract}

We present a geometric approach, in the spirit of the Chern-Weil theory, 
for constructing cocycles representing the classes of the
 Hopf cyclic cohomology of the Hopf algebra H(n) relative to GL(n, R). 
This provides an explicit description of the universal Hopf cyclic Chern classes, 
which complements our earlier geometric realization of the Hopf cyclic characteristic 
classes of foliations.
  \end{abstract}

 \maketitle

 \section*{Introduction}
 
 \smallskip
 
The Hopf algebra $\Hc_n$ originated in the investigation of the local index
formula for transversely hypoelliptic operators on foliations \cite{CM98}, 
performing the role of a `quantum structure group' for foliations of codimension
$n$. Its Hopf cyclic cohomology relative to $\O_n$
was shown to deliver the Gelfand-Fuks cohomology classes as
characteristic classes of `spaces of leaves'. 
 In~\cite{DHC} we presented a geometric method for 
explicitly constructing these universal Hopf cyclic cohomology classes
by means of concrete cocycles, in the spirit of the Chern-Weil theory. We
now supplement that construction by adapting the procedure to the case
of the Hopf cyclic cohomology of  $\Hc_n$  relative to $\GL_n =  \GL_n (\Rb)$, 
which corresponds to the universal equivariant Chern classes.
The essential modification needed to adjust the approach in~\cite{DHC} consists
in the replacement of the `differentiable' variants 
of the standard de Rham complexes for equivariant cohomology  
by a more restrictive version, to be called `regular differentiable'. 
\smallskip

As we often
defer to~\cite{DHC} for additional details, in order to facilitate the reading 
of the present paper we keep the exposition closely parallel to the former.
 In \S \ref{ExpCoc} we introduce the regular differentiable de Rham cohomology 
complexes and use them to 
prove an analogue relative to $\GL_n$ of the van Est-Haefliger isomorphism.
The construction proper of a basis of representative cocycles for 
the Hopf cyclic cohomology of $\Hc_n$ relative to $\GL_n$ is carried out
in \S \ref{UCC}. This
provides a complete description of the universal Hopf cyclic Chern classes, which
complements the geometric realization of the Hopf cyclic characteristic classes
of foliations~\cite{DHC}. Partial representations of these classes 
were obtained earlier by purely algebraic methods
in~\cite[\S 3.4.1]{MR09} (for Hochschild cohomology) and~\cite[\S 4.3]{MR11}.

  \tableofcontents
  
\section{Chern cocycles in regular differentiable cohomology} \label{ExpCoc}

\subsection{Regular differentiable de Rham complexes} \label{EqComp}

Given a manifold $M$ we denote by $\Gb$ 
is the group of diffeomorphisms $\Diff (M)$ 
equipped with the discrete topology, and by
$\triangle_{\Gb} M$
the simplicial manifold $ \{\triangle_{\Gb} M[p] := 
 \Gb^p \ts M\}_{p\geq 0}$ with its usual
face maps $\p_i : \triangle_{\Gb} M[p] \ra \triangle_{\Gb} M[p-1]$,  \, $ 1 \leq i \leq p$, \,   and degeneracies $\,\s_i: \triangle_{\Gb} M[p] \ra \triangle_{\Gb} M[p+1]$,  $\, 0 \leq i \leq p$.
The equivariant cohomology $H_\Gb (M, \Rb)$ can be computed as the cohomology of
the Bott bicomplex (cf. \cite{Bott*, BSS}) 
$\{C^\bullet \left(\Gb, \Om^\bullet (M)\right), \d, d \}$, endowed with the de Rham differential 
$d$ and with the group cohomology boundary  $\d$ 
 \begin{align*} 
 \begin{split}
 \d c (\phi_1, \ldots , \phi_{p+1})\,  = & \sum_{i=0}^{p} (-1)^i  
 c \big(\p_i (\phi_1, \ldots , \phi_{p+1})\big) \\
 & + (-1)^{p+1} 
 \phi_{p+1}^* c (\phi_1, \ldots , \phi_p) .
 \end{split}
 \end{align*}  
For our purposes it will be convenient to work with the homogeneous version of
this bicomplex,
$\{\bar{C}^\bullet \left(\Gb, \Om^\bullet (M)\right), \bar{\d}, d \}$, whose
$(p, q)$-cochains 
$\, \bar{c} (\rho_0 , \ldots , \rho_{p}) \in \Om^q (M)$, $\, \rho_0 , \ldots , \rho_{p} \in \Gb$
satisfy the covariance condition 
\begin{align} \label{cov1}
(\rho^{-1})^* \big(\bar{c} (\rho_0 \rho, \ldots , \rho_{p} \rho)\big) = \bar{c} (\rho_0, \ldots , \rho_p) ,
\quad \forall \, \rho, \rho_i \in \Gb ;
\end{align}
the group cohomology boundary is given by
 \begin{equation*} 
 \bar{\d} \bar{c} (\rho_0, \ldots , \rho_{p}) =
  \sum_{i=0}^{p} (-1)^i \bar{c} (\rho_0, \ldots , \check{\rho_i},
 \ldots , \rho_{p}) , 
\end{equation*} 
where the `check' mark signifies omission of the element.

The two bicomplexes are isomorphic via the identifications
 \begin{align} \label{xchng}
 \begin{split}
  &c (\phi_1, \ldots , \phi_p) = \bar{c} (\phi_1 \cdots \phi_p ,\,  \phi_2 \cdots \phi_p ,\,
  \ldots , \phi_p, e) \, , \\
  \text{resp.} \quad 
 &\bar{c}(\rho_0, \ldots , \rho_p) = \rho_p^* c(\rho_0 \rho_1^{-1}, \rho_1 \rho_2^{-1} ,
 \ldots , \rho_{p-1} \rho_p^{-1}) .
 \end{split}
 \end{align}
 
\smallskip

 Dupont's~\cite{Dupont}
de Rham complex of compatible forms $\{\Om^\bullet (|\triangle_{\Gb} M|), d \}$
on the geometric realization $|\triangle_{\Gb} M| = \prod_{p=0}^\ify \D^p \ts \triangle_{\Gb} M[p]$ 
provides an alternative way
of computing $H^\bullet_{\Gb} (M, \Rb)$.
By definition, such a form 
 consists of sequences $\om = \{\om_p\}_{p\geq 0}$, with
 $\om_p \in \Om^\bullet( \D^p \ts \triangle_{\Gb} M[p])$, such that for all morphisms
 $\mu \in \D(p, q)$ in the simplicial category, 
 \begin{align} \label{sform}
 (\mu_\bullet \ts \Id)^\ast \om_q \, = \, 
 (\Id \ts \mu^\bullet )^\ast \om_p \,\in  \Om^\bullet \left(\D^p \ts \triangle_{\Gb} M[q]\right) ;
\end{align}
here \, $
\D^p = \{\tb = (t_0, \ldots , t_p) \in \Rb^{p+1} \, \mid \, t_i \geq 0, \quad t_0 + \ldots + t_p =1\} $,
$\mu_\bullet  : \D^p \ra \D^q$, resp. $\mu^\bullet  : \triangle_{\Gb} M[q] \ra \triangle_{\Gb} M[p] $,
stands for the induced cosimplicial, resp. simplicial, map, and 
$\Om^k(\D^p \ts \triangle_{\Gb} M[q] $ denotes the $k$-forms on $\D^p \ts \triangle_{\Gb} M[q] $ which
are extendable to smooth forms on $V^p  \ts \triangle_{\Gb} M[q]$, where
$V^p = \{\tb = (t_0, \ldots , t_p) \in \Rb^{p+1} \, \mid \,  t_0 + \ldots + t_p =1\} $.
 As in the case of the previous complex, there is a homogeneous description of the
 simplicial de Rham complex, $\{\Om^\bullet ({|\bar\triangle}_{\Gb} M|), d \}$,
 consisting of the $\Gb$-invariant compatible forms on the geometric realization 
  $|\bar\triangle_{\Gb} M|$. The simplicial manifold $\bar\triangle_{\Gb} M$ 
  is defined as follows:
\begin{align*} 
\bar\triangle_{\Gb} M = \{\bar\triangle_{\Gb} M[p] := 
 \Gb^{p+1} \ts M\}_{p\geq 0} \, ,
\end{align*}
with face maps 
$\bar\p_i : \bar\triangle_{\Gb} M[p] \ra \bar\triangle_{\Gb} M[p-1]$,  \, $ 1 \leq i \leq p$, \,  given by 
\begin{equation*} 
\bar\p_i (\rho_0, \ldots , \rho_p, \, x) \, = \, (\rho_0, \ldots , \check{\rho_i},
 \ldots , \rho_{p}) ,  \quad 0 \leq i \leq p ,
\end{equation*} 
 and degeneracies
$$
\bar\s_i (\rho_0, \ldots , \rho_p, \, x) = (\rho_0, \ldots , \rho_i , \rho_i, \ldots , \rho_p, x) \, ,
\quad 0 \leq i \leq p .
$$
 The compatible forms $\om = \{\om_p\}_{p\geq 0} \in \Om^\bullet ({|\bar\triangle}_{\Gb} M|$
 satisfy the invariance condition
  \begin{align} \label{cov2}
(\rho^{-1})^*\om (\rho_0 \rho, \ldots , \rho_{p} \rho) =  \om (\rho_0, \ldots , \rho_p) ,
\quad \forall \, \rho, \rho_i \in \Gb .
\end{align}

By~\cite[Thm 2.3]{Dupont}, the operation of integration along along the fibers
\begin{align} \label{circint}
\oint_{\D^p}: \Om^{\bullet} (\D^p \ts \triangle_{\Gb} M[p]) \ra 
 \Om^{\bullet -p} (M[p])  \,
\end{align}
establishes a quasi-isomorphism
between the complexes  $\{\Om^\bullet (|\triangle_{\Gb} M|), d \}$ and
$\{ C^{\rm tot} \left(\Gb, \Om^\ast (M)\right), \d \pm d \}$. 
 
\smallskip

Instead of the \textit{differentiable} variants of the above complexes
utilized in~\cite{DHC}, we shall employ here their \textit{regular} versions,
defined as follows.
  
A cochain $\om \in \bar{C}^p \left(\Gb, \Om^q (M)\right)$ will be called 
{\em regular differentiable} if for any local chart $U \subset M$ with coordinates 
$(x^1, \ldots , x^n)$,
\begin{align} \label{difco}
\om (\rho_0, \ldots , \rho_p, x) = 
\sum P_I \left(x, j^k_x(\rho_0), \ldots , j^k_x(\rho_p) \right) dx^I ,
\end{align}
with the functions $P_I$ depending polynomially of a finite number of
jet components of $\rho_a$,  $1 \leq a \leq p$ and of $\big(\det \rho'_a (x)\big)^{-1}$,
where $\rho'_a(x)$ denotes the Jacobian matrix $\big(\p_i\rho^j_a (x)\big)$.
As usual, $dx^I =dx^{i_1} \wg \ldots \wg dx^{i_q}$,
with $I = (i_1< \ldots < i_q)$ running through the set of strictly increasing $q$-indices.
 The cohomology of the total complex 
$\{ \bar{C}_{\rm rd}^{\rm tot} \left(\Gb, \Om^\ast (M)\right), \d + d \}$ 
thus obtained will be denoted  $H_{\rm rd, \Gb}^{\bullet}\left(M, \Rb \right)$.
\smallskip

Similarly, the \textit{regular differentiable} simplicial de Rham complex
 is defined as the subcomplex $\{\Om_{\rm rd}^\bullet (|\bar\triangle_{\Gb} M|), d \}$ of
$\{\Om^\bullet (|\bar\triangle_{\Gb} M|), d \}$  
consisting of the $\Gb$-invariant compatible forms
$\{\om_p\}_{p\geq 0}$ whose components 
satisfy the analogous condition:
 \begin{align} \label{diffo}
 \begin{split}
\om_p (\tb; \rho_0, \ldots , \rho_p, x) = 
\sum P_{I, J} \left(\tb; x, j^k_x(\rho_0), \ldots , j^k_x(\rho_p) \right) dt^I \wg dx^J ,
 \end{split}
\end{align}
with $P_{I, J}$ of the same form as in \eqref{difco}.
We denote by $H_{\rm rd}^\bullet (|\triangle_{\Gb} M|, \Rb)$ the cohomology of the
complex $\{\Om_{\rm rd}^\bullet (|\bar\triangle_{\Gb} M|), d\}$.

\begin{thm} \label{difDup}
The chain map $\displaystyle \oint_{\D^\bullet} : \Om_{\rm rd}^\bullet (|\bar\triangle_{\Gb} M|) \ra 
\bar{C}_{\rm rd}^{\bullet} \left(\Gb, \Om^\ast (M)\right)$ induces an isomorphism
$H_{\rm rd}^\bullet (|\triangle_{\Gb} M|, \Rb) \, \cong H_{\rm rd, \Gb}^{\bullet}\left(M, \Rb \right)$.
\end{thm}

\begin{proof}
The operation of integration along the fibers obviously maps
$\Om_{\rm rd}^\bullet (|\bar\triangle_{\Gb} M|)$ to
$\bar{C}_{\rm rd}^{\bullet} \left(\Gb, \Om^\ast (M)\right)$. The justification
of the parallel result in~\cite[Theorem 2.3]{Dupont} applies here too,
since the natural chain maps in both directions and
the chain homotopies relating them 
preserve the regular differentiable subcomplexes.  
\end{proof}

  \subsection{Van Est-Haefliger isomorphism relative to $\GL_n$} \label{ExpF}
 
For $k \in \Nb \cup \{\infty\}$ we let $F^{k}M$ denote the frame bundle of order $k$,
 formed of $k$-jets $j_0^k(\phi)$ at $0$  
 of local diffeomorphisms $\phi$ from a neighborhood 
 of $0 \in \Rb^n$ to a neighborhood of $\phi(0) \in M$. In particular $F^{1}M = FM$ is the 
 usual principal frame bundle over $M$ with structure group $\Gc^1 = \GL_n$.
 Each $F^{k}M$  is a principal
 bundle over $M$ with structure group $\Gc^k$ formed of $k$-jets at $0$ of 
 local diffeomorphisms of $\Rb^n$ preserving $0$.  
The group $\Gb = \Diff (M)$ operates naturally 
on the left on $F^{k}M$ by left translations.
 
\smallskip
 
Let $\Fa_n$ be the Lie algebra of formal vector fields on $\Rb^n$ and 
denote by $C^{\ast}(\Fa_n)$ its Gelfand-Fuks cohomology complex~\cite{GF}.
Each $\, \om  \in C^{m}(\Fa_n)$ gives rise to a  $\Gb$-invariant
 form $\tilde\om \in \Om^m (F^{\ify}M)$, and the assignment 
$\om \in C^\bullet(\Fa_n) \mapsto \tilde\om \in \Om^\bullet (F^{\ify}M)^\Gb$
is a DGA-isomorphism, by means of which we shall identify the two DG-algebras.
 
 \smallskip
  
After fixing a torsion-free affine connection $\nabla$ on $M$, we define 
a cross-section $\s_{\nabla} : FM \ra F^{\ify}M$
of the natural projection $\pi_1 : F^\ify M \ra FM$
by the formula
\begin{equation} \label{jcon}
\s_{\nabla} (u) = j_0^{\ify} (\exp_x^{\nabla} \circ u) \ , \qquad u \in F_x M \, .
\end{equation}
Clearly,  $\s_{\nabla}$ is ${\rm GL}_n$-equivariant
 and Diff-equivariant:
\begin{equation} \label{nat}
\s_{\nabla^\phi} = {\phi}^{-1} \circ \s_{\nabla}  \circ {\phi} \ , 
\qquad \fl \, \phi \in \Gb  \, ;
\end{equation}
here $\nabla^\phi = \phi_\ast^{-1} \circ \nabla \circ \phi_\ast$, with
connection form ${\phi}^* \om$. 

For each $p \in \Nb$,
we define $\s_p : \D^p \ts \bar\triangle_{\Gb} FM[p] \ra F^\ify M$ by  
\begin{align} \label{xchng2}
\begin{split}
\s_p (\tb ; \rho_0 , \ldots , \rho_p, u) &= \,
\s_{\nabla (\tb; \rho_0 , \ldots , \rho_p)} (u)  ,\\
\text{where}   \quad
 \nabla (\tb; \rho_0 , \ldots , \rho_p) &= \, \sum_0^p \ t_i \, \nabla^{\rho_i} , \qquad
 \tb \in \Delta^p .
\end{split}
\end{align}
The collection $\, \hat{\s} = \{\s_p \}_{p \geq 0}$ descends to the geometric realization of
 $\bar\triangle_{\Gb} FM$, giving a map 
 $\, \hat{\s} : |\bar\triangle_{\Gb} FM| \ra F^\ify M$. By construction, $\hat{\s}$ is
 $\GL_n$-equivariant and therefore it also induces a map 
$\, \hat{\s}^{\GL_n} : |\bar\triangle_{\Gb} M| \ra F^\ify M/\GL_n$. 
 
\begin{lem} \
 If $\om \in C^\bullet(\Fa_n)$ then 
 $\hat{\s}^*(\tilde\om) \in \Om_{\rm rd}^\bullet (|\bar\triangle_{\Gb} FM|$.
  \end{lem}
 
 \begin{proof} First we note that, because
$\tilde\om$ is $\Gb$-invariant, $\hat{\s}^*(\tilde\om)$ is 
easily seen to be a compatible form. 
 It remains to check that for any $ \phi \in \Gb$ and any local chart $U$, with the
 notation as in \eqref{difco}, one has
\begin{align*}
 \s_{\nabla^\phi}^*(\tilde\om)(x)= 
 \sum P_I \left(x, j^k_x(\phi)\right) dx^I , \qquad x \in U .
 \end{align*}
Using normal coordinates with respect to $\nb$, this follows from the explicit
expression for $\s_{\nabla^\phi}$ in the proof of Lemma 3.5 in~\cite{DHC}.
 \end{proof}
 
In view of the above lemma, it makes sense to define
 $\,  \Cc_\nabla :  C^\bullet(\Fa_n) \ra \Om_{\rm rd}^\bullet (|\bar\triangle_{\Gb} FM|)$
 by
\begin{align} \label{crux}
 \Cc_\nabla (\om )\, = \, 
  \hat{\s}^*(\tilde\om) \in \Om_{\rm rd}^\bullet (|\bar\triangle_{\Gb} FM|) .
\end{align}

The map $\,  \Cc_\nabla$ is a homomorphism of DG-algebras. which in turn
induces a DGA-homomorphism 
at the level of $\GL_n$-basic forms,
\begin{align} \label{crux2}
 \Cc^{\GL_n}_\nabla: C^\bullet(\Fa_n, \GL_n) \ra 
 \Om_{\rm rd}^\bullet  (|\bar\triangle_{\Gb} M|) .
 \end{align}
 
 \begin{thm} \label{main1}
 The map $\, \Cc^{\GL_n}_\nabla$ is a quasi-isomorphism of DG-algebras.  
   \end{thm} 
  
 \begin{proof}  The proof follows along the same lines as that of \cite[Theorem 1.4]{DHC}.
For any connection $\td\nabla$, one has
 \begin{equation*} 
(\pi_1 \circ \s_{\td\nabla}) (u) = j_0^{1} (\exp_x^{\td\nabla} \circ u) = u, \quad u \in F_x M \, .
\end{equation*}
After upgrading $\pi_1$ and $\hat{\s}$ to simplicial maps
$ \Id \ts \pi_1 : |\bar\triangle_{\Gb} F^\ify M| \ra |\bar\triangle_{\Gb} FM|$ and
$\Id \ts \hat{\s}  : |\bar\triangle_{\Gb} FM| \ra |\bar\triangle_{\Gb} F^\ify M|$,
one obtains
\begin{align*} 
(\Id \ts \pi_1) \circ (\Id \ts \hat{\s})  \, =\, \Id .
\end{align*}
Hence $ (\Id \ts \hat{\s})* : \Om_{\rm rd}^\bullet (|\bar\triangle_{\Gb} F^\ify M|) \ra
 \Om_{\rm rd}^\bullet (|\bar\triangle_{\Gb} FM|)$
is a left inverse for 
$(\Id \ts \pi_1)^* : \Om_{\rm rd}^\bullet (|\bar\triangle_{\Gb} FM|) \ra 
\Om_{\rm rd}^\bullet (|\bar\triangle_{\Gb} F^\ify M|)$. Both maps 
are $\GL_n$-equivariant and thus descend to maps 
\begin{align*} 
\begin{split}
 (\Id \ts \hat{\s})_{\GL_n}^* &: \Om_{\rm rd}^\bullet (|\bar\triangle_{\Gb} F^\ify M/\GL_n|) \ra
 \Om_{\rm rd}^\bullet (|\bar\triangle_{\Gb} M|), \\
 \text{resp.} \quad
(\Id \ts \pi_1)_{\GL_n}^* &: \Om_{\rm rd}^\bullet (|\bar\triangle_{\Gb} M|) \ra 
\Om_{\rm rd}^\bullet (|\bar\triangle_{\Gb} F^\ify M/\GL_n|) .
\end{split}
\end{align*}
The typical fiber $\Gc^k/\GL_n$ of $F^\infty M/\GL_n \ra M $ can be canonically identified 
to the pronilpotent group $\Gc^k_1$ of $\infty$-jets at $0$ of 
 local diffeomorphisms of $\Rb^n$ preserving $0$ to order $1$. As such,
 it is  {\em algebraically} contractible, hence
 $(\Id \ts \pi_1)_{\GL_n}^*$ induces an isomorphism in regular differentiable cohomology.
 Therefore so does its inverse $ (\Id \ts\hat{\s})_{\GL_n}^*$.
 \smallskip
 
 On the other hand, identifying
 the $\GL_n$-basic forms on $F^\ify M$ with forms on 
 $P^\ify M = F^\ify M/\GL_n$, one defines a
 horizontal homotopy as in~\cite[Lemma 2.3]{Kumar},
 by the formula
 \begin{align*}
 \begin{split}
 (H\a)_{p-1}(\tb; &\rho_0, \ldots, \rho_{p-1} , \, j_0^\ify(\rho) \GL_n) = \\
&\pi_{\GL_n} [k \in \GL_n
\mapsto \a_p (\tb; (\rho k)^{-1}, \rho_0, \ldots, \rho_{p-1} , \, j_0^\ify(\rho) \GL_n)] ,
 \end{split}
 \end{align*}
 where $\pi_{\GL_n}$ stands for the projection on the  $\GL_n$-invariant (constant)
part with respect to the decomposition into isotypical components
of the right regular representation of $\GL_n$ on its ring of 
regular functions tensored by the fiber.

Therefore the natural inclusion
 of $C^\bullet(\Fa_n, \GL_n) \equiv \Om^\bullet (F^\ify M/\GL_n)^\Gb$ into  
 $\Om_{\rm rd}^\bullet (|\bar\triangle_{\Gb} F^\ify M/\GL_n|)$ 
is also quasi-isomorphism. To complete the
proof it remains to observe that
 when restricted to $\GL_n$-basic forms the map
 $ (\Id \ts\triangle \s )^*$ coincides with $\Cc_\nabla^{\GL_n}$.
  \end{proof}

Combining the   Theorems \ref{difDup} and \ref{main1}
one obtains the `relative to $\GL_n$' version of
the van Est-Haefliger isomorphism~\cite[\S IV.4]{HaefDC}.

\begin{thm} \label{main2}
 The map  
\begin{align*}
 \Dc^{\GL_n}_\nabla = \oint_{\D^\bullet} \Cc^{\GL_n}_\nabla
 : C^\bullet(\Fa_n, \GL_n) \ra 
\bar{C}_{\rm rd}^{\rm tot \, \bullet} \left(\Gb, \Om^\ast (M)\right)
\end{align*}
is a quasi-isomorphism of complexes. 
\end{thm}
\smallskip

\subsection{Equivariant Chern cocycles} \label{ECC}

Let $\, W(\Fg\Fl_n) = \wg^\bullet \Fg\Fl^*_n \ot S(\Fg\Fl_n)$  be the Weil algebra
of $\Fg\Fl_n$ with its usual grading, and let
$\,\hat{W}(\Fg\Fl_n) = W(\Fg\Fl_n)/ \Ic_{2n}$ be its truncation by
the ideal generated by the elements of $S(\Fg\Fl_n) $ of degree $> 2n$.   
 The universal connection and curvature forms 
 $\vartheta = (\vartheta^i_j)$ and $R = (R^i_j)$, defined as in
 \cite[\S 2]{Bott*}, generate
 a DG-subalgebra $CW^\bullet (\Fa_n)$ of $C^\bullet (\Fa_n)$,
 which can be identified with $\,\hat{W}(\Fg\Fl_n)$. Let
  $CW^\bullet (\Fa_n, \GL_n)$, resp. $\hat{W}(\Fg\Fl_n,  \GL_n)$, denote
  their subalgebras consisting of $ \GL_n$-basic elements,
  also identified as above. 
  It follows from Gelfand-Fuks~\cite{GF} (cf. also \cite{Godbillon}) that the inclusion of
the latter into $C^\bullet (\Fa_n, \GL_n)$ is a quasi-isomorphism.
Thus, by Theorems \ref{main1} and \ref{main2},
\begin{align} \label{UCW1}   
 \Dc^{\GL_n}_\nabla : \hat{W}(\Fg\Fl_n,  \GL_n) \equiv CW^\bullet(\Fa_n, \GL_n) \ra 
\Om_{\rm rd}^\bullet  (|\bar\triangle_{\Gb} M|) 
\end{align} 
is a DGA quasi-isomorphism and
\begin{align} \label{UCW2}   
 \Dc^{\GL_n}_\nabla : \hat{W}(\Fg\Fl_n,  \GL_n) \equiv CW^\bullet(\Fa_n, \GL_n) \ra 
\bar{C}_{\rm rd}^{\rm tot \, \bullet} \left(\Gb, \Om^\ast (M)\right)
\end{align} 
 is a quasi-isomorphism of complexes.
 
The cohomology of $\hat{W}(\Fg\Fl_n,  \GL_n)$ is well-known to be
isomorphic to the truncated polynomial ring 
 generated by the universal Chern classes $P_{2n} [c_1, \ldots, c_n]$, 
 with $c_1, \ldots, c_n$
given by the invariant polynomials
\begin{align} \label{uchern}
 c_q (A) \, = \, \sum_{1\le i_1 < \ldots < i_q \le n} \sum_{\mu \in S_q} 
 (-1)^\mu A^{i_1}_{\mu(i_1)} \cdots A^{i_q}_{\mu(i_q)} ,
  \quad A \in \Fg \Fl_n.
\end{align}
The above quasi-isomorphisms allow to transport the standard
basis of $P_{2n} [c_1, \ldots, c_n]$ to a basis of 
$H_{\rm rd}^{\bullet} \left(\Gb, \Om^\ast (M)\right)$, as follows.

Let   $  \om_\nb = (\om_j^i)$, resp.  $\Om_\nb = (\Om_j^i )$, denote the matrix-valued 
connection form, resp. curvature form, corresponding to $\nb$.
One has the naturality relation (cf.~\cite[Lemma 18]{CM01}),
 \begin{align}\label{Luc}
\s_{\nabla}^* (\wt{\vartheta}_j^i) = \om_j^i \qquad \text{hence} \qquad
\s_{\nabla}^* (\wt{R}_j^i) = \Om_j^i .
\end{align}
   In homogeneous group coordinates (cf. \eqref{xchng}), the
  simplicial connection form-valued matrix
   $\hat{\om_\nb} = \{ \hat\om_p \}_{p \in \Nb}$ associated to $\nabla$
  has components
\begin{align} \label{scone}
 \hat\om_p (\tb ; \rho_0, \ldots , \rho_p) : = \sum_{i=0}^p t_i \rho_i^* (\om_\nb) ,  
\end{align}
and the 
 simplicial curvature form-valued matrix 
 $\hat{\Om_\nb} : = d \hat{\om_\nb} + \hat{\om_\nb} \wg \hat{\om}$
has components \, $\hat{\Om}_p = \hat{\Om}_p^{(1, 1)} + \hat{\Om}_p^{(0, 2)}$,
given by
\begin{align}  \label{scurv}
\begin{split}
&\hat{\Om}_p (\tb ; \rho_0, \ldots , \rho_p) \, = \, \sum_{i=0}^p dt_i \wdg \rho_i^* (\om_\nb) \, + \\
&\sum_{i=0}^p t_i \big(\rho_i^* (\Om_\nb) - \rho_i^* (\om_\nb) \wdg \rho_i^* (\om_\nb)\big) 
+ \sum_{i, j=0}^p t_i t_j \,  \rho_i^* (\om_\nb) \wdg  \rho_j^* (\om_\nb) .
\end{split}
 \end{align} 
The forms $\hat{\om}_j^i$ and $\hat{\Om}_j^i$ clearly belong to the 
regular differentiable de Rham complex $\Om_{\rm rd}^\bullet (|\bar\triangle_{\Gb} FM|)$.
In addition, the Chern forms $c_k (\hat\Om_\nb)$ are $\GL_n$-basic and therefore
descend to
$ \Om_{\rm rd}^{2k} (|\triangle_{\Gb}M|)$, and we denote by the same symbols
the corresponding cohomology classes. 
In view of the DGA quasi-isomorphism \eqref{UCW1}
the cohomology ring $H_{\rm rd}^\bullet  (|\bar\triangle_{\Gb} M|, \Cb)$ is
isomorphic to $P_{2n} [c_1, \ldots, c_n]$. Therefore the collection of forms
\begin{align} \label{chern-forms}
  c_J (\hat{\Om}_\nb) \, = \, 
   c_{j_1}(\hat{\Om}_\nb) \wg  \ldots \wg c_{j_q} (\hat{\Om}_\nb)
  \in \Om_{\rm rd}^{2|J|} (|\triangle_{\Gb}M| ,
\end{align}
with $J= (j_1 \leq \ldots \leq j_q)$ and $|J| := j_1 +\ldots + j_q \leq n$, 
represents a linear 
basis of $H_{\rm rd}^\bullet  (|\bar\triangle_{\Gb} M|, \Cb)$. Applying now
the quasi-isomorphism \eqref{UCW2} (which is linear, but not multiplicative)
one obtains representative cocycles for
a linear basis of $H_{\rm rd, \Gb}^\bullet  (M, \Cb)$, namely
\begin{align} \label{basis-forms}
  C_J (\hat{\Om}_\nb) \, := \, \oint_{\D^\bullet} c_J (\hat\Om_\nb) ,
  \quad  J= (j_1 \leq \ldots \leq j_q),\, \,  |J| \leq n  \} .
  \end{align}
 
\section{Hopf cyclic universal Chern classes}  \label{UCC}
 
\subsection{Hopf algebra $\Hc_n$ and its Hopf cyclic complex}  \label{Hn}

The Hopf algebra $\Hc_n $ arises quite naturally
as the symmetry structure of the convolution algebra 
$C_c^{\ify}  (\bar{\G}_n)$ of the \'etale groupoid $\bar{\G}_n$ of germs of 
local diffeomorphisms of $\Rb^n$ acting by prolongation on the frame bundle 
$F\Rb^n$, identified with the affine group $G=\Rb^n \ltimes \GL_n$.
Equivalently, it acts naturally on the crossed product algebras 
$ {\Ac_\Gamma} = C_c^{\ify} (F\Rb^n ) \rtimes \Gamma$, with
$\Gamma$ a discrete subgroup of $\Gb = \Diff {\Rb}^n$. We briefly review below
its operational construction, and refer the reader
to \cite{DHC} for a more detailed account.

The primary generators of $\Hc_n $ are the (horizontal, resp. vertical) 
left-invariant vector fields $\{ X_k, Y_i^j \mid i, j, k =1, \ldots , n\} $, 
that form the standard basis  of the Lie algebra $\Fg = \Rb^n \ltimes \Fg \Fl_n$ of $G$. 
The vector fields $Z \in \Fg$ are made to
act on the algebra $\Ac := C_c^{\ify} (F\Rb^n ) \rtimes \Gb $ by 
\begin{equation*}
Z (f \, U_{\vp}) \, = \,Z( f) \, U_{\vp} , \qquad f
\, U_{\vp}^* \in {\Ac} \, ,
\end{equation*}
the resulting linear operators on $\Ac$ satisfy generalized Leibnitz rules, which in the 
Sweedler notation take the form
\begin{equation*}
Z (a \, b) \, = \, Z_{(1)} (a) \, Z_{(2)}(b)  ,
\quad a, b \in {\Ac} .
\end{equation*}
In particular,  
\begin{equation*} 
X_k(a \, b) \, = \, X_k (a) \, b \, + \, a \,X_k (b) \, + \,
\d_{jk}^i (a) \, Y_{i}^{j} (b) \, ,
\end{equation*}
where  
\begin{align} \label{gijk}
\begin{split}
&\d_{jk}^i (f \, U_{\vp^{-1}}) \, =\, \g_{jk}^i  (\phi) \, f \,U_{\vp^{-1}} , \quad \text{with} \\
&\g_{jk}^i (\phi) (x, {\bf y}) \, =\, \left( {\bf y}^{-1} \cdot
{\phi}^{\prime} (x)^{-1} \cdot \part_{\mu} {\phi}^{\prime} (x) \cdot
{\bf y}\right)^i_j \, {\bf y}^{\mu}_k \, .
\end{split}
\end{align}
The operators $\, \d_{jk}^i $ are derivations, but their
successors
$\, \d_{jk \,  \ell_1 \ldots \ell_r}^i = [X_{\ell_r} , \ldots
[X_{\ell_1} , \d_{jk}^i] \ldots ]$,
 \begin{align} \label{highg}
\begin{split} 
&\d_{jk \,  \ell_1 \ldots \ell_r}^i \, ( f \, U_{\vp^{-1}}) =
 \g_{jk \,  \ell_1 \ldots \ell_r}^i (\phi)\, f \, U_{\vp^{-1}}\qquad   \text{where}\\ 
 & \g_{jk \,  \ell_1 \ldots \ell_r}^i (\phi) =
 X_{\ell_r} \cdots X_{\ell_1} \big(\g_{jk}^i (\phi)\big)  \, , \quad \phi \in \Gb ,
 \end{split}
\end{align}
obey progressively more elaborated Leibnitz rules.
 The subspace $\Fh_n$ of linear operators on $\Ac$ generated by the operators
$X_k$, $Y^i_j$, and  $\d_{jk \,  \ell_1  \ldots \ell_r}^i$
forms a Lie algebra $\Fh_n$. 

By definition, $\Hc_n$ is the algebra of linear operators on $\Ac$
generated by $\Fh_n$ and the scalars.
For $n >1$ the operators $\, \d_{jk \,  \ell_1 \ldots \ell_r}^i \,$
are not all distinct. They satisfy the ``structure identities''
\[
 \d_{j \ell \,  k}^i \, - \,  \d_{j k \,  \ell}^i \, = \,
 \d_{j k}^s \, \d_{s \ell}^i  \ - \, \d_{j \ell}^s \, \d_{s k}^i \, ,
\]
reflecting the flatness of the standard connection.
The algebra $\Hc_n$ is isomorphic to the quotient 
$\FA ({\Fh}_n)/\Ic$ of the universal enveloping
algebra $\FA ({\Fh}_n)$ by the ideal $\Ic$ generated by the above identities. It has
a distinguished character  $\d :\Hc_n \ra \Cb$, which 
 extends the modular character of
 ${\Fg \Fl}_n (\Rb)$, and is induced from the
 character of ${\Fh}_n$ defined by
 \[
\d(Y_i^j) = \d_i^j , \quad  \d(X_k) = 0, \quad \d(\d_{jk \,  \ell_1 \ldots \ell_r}^i) =0 .
 \]

The coproduct of $\Hc_n$  
stems from the interaction of $\Hc_n$ with the
product of $\Ac$. More precisely, any $h \in \Hc_n$ satisfy an identity of the form
\[
 h (a  b) \ = \ \sum_{(h)} \, h_{(1)} (a) \, h_{(2)} (b) \, ,
\quad  h_{(1)}, \, h_{(2)} \in \Hc_n , \quad a, b \in \Ac \, ,
\]
and this uniquely determines a
{\it coproduct} $\D : \Hc \ra \Hc_n \ot \Hc_n$, by setting (using Sweedler's notation)
\begin{equation*} 
\D ( h) \ = \ \sum_{(h)} \, h_{(1)} \ot h_{(2)} .
\end{equation*}
 The {\it counit} is \, $ \ve( h)  \, = \,  h (1)$, while the antipode $S$ is uniquely
 determined by its very definition 
 \[
\sum_{(h)} S(h_{(1)})  h_{(2)}  \ = \ \ve (h)\cdot1 \ = \ \sum_{(h)} h_{(1)} S(h_{(2)}) .
\]
Although the antipode $S$ fails to be involutive, its twisted version
 \[
 S_\d (h) \ = \  \sum_{(h)} \d(h_{(1)}) S(h_{(2)}) 
 \]
 does satisfy the property
 \begin{align} \label{invant}
S_\d^2 \, = \, \Id .
\end{align}

The algebra $ {\Ac}$ has a canonical trace, namely
 \begin{equation} \label{tr}
\tau \, (f \, U_{\vp})  \, = \, \left\{ \begin{matrix}
\displaystyle
  \int_{F\Rb^n} \, f  \, \varpi \, , \quad \text{if}
\quad \vp = \Id \, , \cr\cr \quad 0 \, , \qquad \qquad
\text{otherwise} \, ;
\end{matrix} \right.
\end{equation}
here $\varpi$  is the volume form determined by the dual to the 
canonical basis of $\Fg$.
This trace satisfies
\begin{equation} \label{it}
\tau (h(a)) \, = \,  \d(h)\, \tau(a) , \qquad h \in \Hc_n, \,  \, a \in \Ac\, .
\end{equation}

 The standard Hopf cyclic model for $\Hc_n$ is imported from the standard cyclic
model of the algebra $\Ac$, by means of the characteristic map
\begin{align} \label{char-map}
\begin{split} 
&h^1 \ot \ldots \ot h^q \in \Hc_n^{\ot^q}  \longmapsto
 \chi_{\tau} (h^1 \ot \ldots \ot h^q) \in C^q (\Ac)
\, , \\  
&\chi_{\tau} (h^1 \ot \ldots \ot h^q) (a^0 , \ldots , a^q) = 
 \tau (a^0  h^1 (a^1) \ldots h^q (a^q)) , \quad a^j \in \Ac ,
\end{split}
\end{align}
It gives rise to a cyclic structure~\cite{Cext} on 
$\, \{ C^q (\Hc_n ; \delta) := \Hc_n^{\ot^q} \}_{q \geq 0}$, with faces, degeneracies and cyclic
 operator given by
\begin{eqnarray*} 
\d_0 (h^1 \ot \ldots \ot h^{q-1}) &=& 1 \ot h^1
\ot \ldots \ot h^{q-1} , \\  
\d_j (h^1 \ot \ldots \ot h^{q-1}) &=& h^1 \ot \ldots \ot \D h^j \ot
\ldots \ot h^{q-1}, \quad 1 \leq j \leq q-1 , \\  
\d_n (h^1 \ot \ldots \ot h^{q-1}) &=& h^1 \ot \ldots \ot h^{q-1}
\ot 1; \\
 \s_i (h^1 \ot \ldots \ot h^{q+1}) &=& h^1 \ot \ldots \ot \ve
(h^{i+1}) \ot \ldots \ot h^{q+1} , \quad 0 \leq i \leq q \, ; \\  
   \tau_q (h^1 \ot \ldots \ot h^q) &=& S_\d (h^1) \cdot (h^2
\ot \ldots \ot h^q \ot 1) \, .
\end{eqnarray*}
The identity $\, \tau_q^{q+1} = \Id \,$ is satisfied precisely
 because of the involutive property \eqref{invant}, to which is actually equivalent.
\smallskip

The periodic Hopf cyclic cohomology $HP^\bullet (\Hc_n; \Cb_\d)$
of $\Hc_n$ with coefficients in the modular pair $(\d, 1)$ is, by definition
(cf. \cite{CM98, CM99}), the $\Zb_2$-graded cohomology of the total complex 
$\, CC^{\rm tot} (\Hc_n ; \Cb_\d)$ associated to
the bicomplex 
$\, \{C C^{*, *} (\Hc_n ;\Cb_\d), \, b , \, B \}$, where
 \begin{align*}  
b = \sum_{k=0}^{q+1} (-1)^k \d_k , \qquad
  B = ( \sum_{k=0}^q (-1)^{q \, k}\tau_q^k ) \sigma_{q-1}\tau_q \, .
\end{align*} 

The periodic Hopf cyclic cohomology of $\Hc_n$ relative to $\GL_n$,  denoted
$HP^\bullet (\Hc_n, \GL_n ; \Cb_\d)$, is the cohomology of the cyclic complex
defined as follows. One considers the quotient
 $\Qc_n:= \Hc_n \ot_{\Uc (\Fg \Fl_n)} \Cb \equiv \Hc_n /\Hc_n \Uc^+ (\Fg \Fl_n)$, which is
 an $\Hc_n$-module coalgebra with respect to the coproduct and counit
 inherited from $\Hc_n$. One then forms the cochain complex
 \begin{align*} 
C^q  (\Hc_n, \GL_n ; \Cb_\d) : = \, 
\Cb_\d \ot_{\Uc (\Fg \Fl_n)} \Qc_n^{\ot q} \equiv \left(\Qc_n^{\ot q}\right)^{\GL_n} ,
\qquad q \geq 0 ,
\end{align*} 
endowed with the cyclic structure given by restricting to $\GL_n$-invariants the operators
 \begin{eqnarray*}  
 \d_0 (c^1 \ot \ldots \ot c^{q-1}) &=& 
  \dot{1} \ot c^1 \ot \ldots \ot
  \ldots \ot c^{q-1} ,   \\  
\d_i  (c^1 \ot \ldots \ot c^{q-1})  &=& 
c^1 \ot \ldots \ot \D c^i \ot \ldots \ot c^{q-1} , \quad 1 \leq i \leq q-1; \\ 
\d_n  (c^1 \ot \ldots \ot c^{q-1})  &=& c^1 \ot \ldots \ot c^{q-1}
\ot \dot{1}\, ; \\  
 \s_i  (c^1 \ot \ldots \ot c^{q+1})  &=& c^1 \ot \ldots \ot \ve
(c^{i+1}) \ot \ldots \ot c^{q+1} ,\quad 0 \leq i \leq q \, ; \\   
   \tau_q ( \dot{h}^1 \ot c^2 \ot \ldots \ot c^q) &=& 
    {S_\d}(h^1) \cdot (c^2
   \ot \ldots \ot c^q \ot \dot{1}).
\end{eqnarray*}
 The corresponding characteristic map lands in the cyclic cohomology 
 of the crossed product algebra $\Ac_{\rm base} = C_c^\infty (\Rb^n) \rtimes \Gb$,
 and is given at the chain level by the map
 $c \in \left(\Qc_n^{\ot q}\right)^{\GL_n}   \mapsto
 \chi_{\rm base} (c) \in C^q (\Ac_{\rm base})$ 
 defined as follows: 
\begin{align} \label{rel-char-map}
\chi_{\rm base} (\dot{h}^1 \ot \ldots \ot \dot{h}^q) (a^0 , \ldots , a^q) = 
 \tau_{\rm base} (\td a^0  h^1 (\td a^1) \ldots h^q (\td a^q)) ,
\end{align}
where $ \tau_{\rm base}$ is the canonical trace of $ \Ac_{\rm base}$,
 $\dot{h}$ stands for the class in $\Qc_n$ of $h \in \Hc_n$, and for a
monomial $a = f U_\phi \in \Ac_{\rm base}$ we let
$\td a := \td f U_\phi \in \Ac$, with
$\td f  \in C^\infty (F\Rb^n)$ denoting the lift of $f \in C_c^\infty (\Rb^n)$ via
the natural projection $F\Rb^n \ra \Rb^n$. The definition 
makes sense, as it can be checked that the element
$\td a^0  h^1 (\td a^1) \ldots h^q (\td a^q) \in \Ac$ is independent of the representatives
$h^i$ of the classes $\dot h^i$, and
does descend to $\Ac_{\rm base}$.

\smallskip

  \subsection{From equivariant to Hopf cyclic cohomology}  \label{LtoH}

We recall the definition of the map $\Phi$ of Connes~\cite[III.2.$\d$]{book},
specialized to the present context.
Consider the DG-algebra,
$\, {\Bc_\Gb} (G)= \Om^*_c(G) \ot \wg \, \Cb [\Gb']$, where $\Gb' = \Gb \setminus \{e\}$,
with the differential $d \ot \Id$.
One labels the generators of $\Cb [\Gb']$ as
$\g_{\phi}$, $\phi \in \Gb$, with $\g_e = 0$, and one forms the crossed product
$\, \Cc_\Gb (G)  = {\Bc_\Gb}(G) \rtimes \Gb$, 
with the commutation rules
\begin{align*}
&U_{\phi}^\ast \, \om \, U_{\phi} = \phi^\ast \, \om  , &\qquad
\, \om \in \Om^*_c(G),\\
& U_{\phi_1}^\ast \, \g_{\phi_2} \, U_{\phi_1} =\g_{\phi_2 \circ \phi_1} -
\g_{\phi_1} , &\qquad  \phi_1 , \phi_2 \in \Gb \, .
\end{align*}
$\Cc_\Gb (G) $ is also a DG-algebra, equipped with the differential  
\begin{equation} \label{dbo}
{\dbo} (b \, U_{\phi}^\ast) = db \, U_{\phi}^\ast - (-1)^{\p b} \, b \, \g_{\phi} \,
U_{\phi}^\ast  , \qquad b \in {\Bc_\Gb} (G) , \quad \phi \in \Gb,
\end{equation}

A cochain $\lambda \in \bar{C}^{q}(\Gb, \Om^p(G))$ determines a linear form
$\wt{\lambda}$ on $\Cc_\Gb (G) $ as follows: 
\begin{align}   \label{prePhi}
\begin{split}
&\wt{\lambda} (b \,U_{\phi}^\ast) = 0 \qquad  \text{for} \quad  \phi \ne \one ; \\
& \text{if}  \quad \phi = \one \quad  \text{and} \quad
 b=\om \ot \g_{\rho_1} \ldots \g_{\rho_q} \qquad  \text{then} \\
&\wt{\lambda}(\om \ot \g_{\rho_1} \ldots \g_{\rho_q}) = \int_{G}
 \lambda(1, \rho_1 , \ldots ,\rho_q) \wg \om .
   \end{split}
\end{align}
The map $\Phi$ from $ \bar{C}^{\bullet}(\Gb, \Om^\bullet(G))$
to the $(b, B)$-complex of the algebra $\Ac = C_c^\ify (G)  \rtimes \Gb$ 
is now defined for  $\lambda \in \bar{C}^{q}(\Gb, \Om^p(G))$ by
\begin{align}   \label{mapPhi}
\begin{split}
\Phi(\lambda)(a^0, \ldots, a^m)&=
\frac{p!}{(m+1)!}
\sum_{j=0}^l(-1)^{j(m-j)}\wt{\lambda}({\dbo}a^{j+1}\cdots {\dbo}a^m\; a^0\; 
{\dbo}a^1\cdots {\dbo}a^j) \\
   \text{where} \quad m &=\dim G-p+q  , \qquad a^0, \ldots, a^m \in \Ac .
   \end{split}
\end{align}
By~\cite[III.2.$\d$, Thm. 14]{book}, $\Phi$ is a chain map to the total
 $(b, B)$-complex of the algebra $\Ac$.
 \smallskip

The relative version $\Phi^{\GL_n}$ of the map $\Phi$ is obtained by first
replacing $\Om^*_c(G) $ with the $\GL_n$-basic forms $\Om^*_{c, \rm basic}(G)$
which are compact modulo $\GL_n$, and so can be identified to $\Om^*_c(\Rb^n)$,
and then replacing in the definition \eqref{prePhi} the integration over 
$G=\Rb^n \ltimes \GL_n$ by integration over the base $\Rb^n$. One obtains this way
the induced chain map
\begin{align} \label{rel-chain-map}
\Phi^{\GL_n} : \bar{C}^{\bullet}(\Gb, \Om^\bullet(\Rb^n)) \ra  C^\bullet (\Ac_{\rm base}) .
 \end{align}
 
Assume now that $\lambda \in \bar{C}^{q}(\Gb, \Om^p(\Rb^n))$ is of the form 
$\lambda = \Dc_\nb  (\om)$ with  $\om \in C(\Fa_n, \GL_n)$,
 where $\nb$ stands for the standard flat connection. Using \cite[Lemma 3.5]{DHC}
 which identifies the map $\Dc_\nb$ with the map $\Dc$ employed 
 in~\cite{CM98}, one shows as in~\cite[pp. 233-234]{CM98}) that 
$\Phi^{\GL_n}(\lb)$ has the expression  
\begin{align}  \label{Phim}
 \Phi^{\GL_n}(\lambda)(a^0, \ldots, a^q) \, = \,
 \tau_{\rm base} (\td a^0  h^1 (\td a^1) \ldots h^q (\td a^q)) ,
\end{align}
 with $\, \sum_\a  \dot h_\a^1 \ot \ldots \ot \dot h_\a^q \in \left(\Qc_n^{\ot \, q}\right)^{\GL_n}$ 
uniquely determined by $\lb$. This means that $\Phi^{\GL_n}(\lb)$ lands in the 
$(b, B)$-complex which defines the Hopf cyclic cohomology of $ \Hc_n $ relative to
$\GL_n$.  
Thus, by restricting $\Phi^{\GL_n}$ to the subcomplex 
\begin{align}  \label{sups}
\bar{C}_{\Dc}^{\rm tot}(\Gb, \Om^*(\Rb^n)) :=\Dc_\nb \big(C(\Fa_n, \GL_n)\big) 
\subset \bar{C}_{\rm rd}^{\rm tot}(\Gb, \Om^*(\Rb^n)) ,
\end{align}
one obtains a chain map
\begin{align}  \label{ups}
 \Phi^{\GL_n}_{\rm rd}: \bar{C}_{\Dc}^{\rm tot}(\Gb, \Om^*(\Rb^n)) \ra 
 CC^{\rm tot} (\Hc_n, \GL_n ; \Cb_\d) .
 \end{align}
 By~\cite[Theorem 11]{CM98}, or more precisely its relative to $\GL_n$ version, the
 composition $\Phi^{\GL_n}_{\rm rd} \circ \Dc^{\GL_n}_\nb$ is  a quasi-isomorphism.
 Since, by construction (cf. \S \ref{ECC}) the cocycles $C_J (\hat{\Om}_\nb)$
 are images via the map $\Dc^{\GL_n}_\nb$ of representatives for a basis of 
 $H^*(\Fa_n, \GL_n)$,  we can finally conclude that:
 
 \begin{thm} \label{vE}
The collection of cocycles
 \begin{align*} 
  \{\Phi^{\GL_n}_{\rm rd}\big(C_J (\hat{\Om}_\nb)\big) \, ; \qquad
  J= (j_1 \leq \ldots \leq j_q),\quad  |J| \leq n  \}
  \end{align*}
 represent a basis of $HP^\bullet (\Hc_n, \GL_n ; \Cb_\d)$.  
 \end{thm}

 To get more insight into the makeup of these cocycles, we recall that
 $\nb$ is the flat connection on
$G \equiv F\Rb^n \ra {\Rb}^n$, so its
 connection form  is\, $\om_\nb = \left(\om^i_j \right)$ with
$\,  {\om}^i_j \, :=  \, ({\bf y}^{-1})^i_{\mu} \, d{\bf y}^{\mu}_j = \big({\bf y}^{-1} \, d{\bf y}\big)^i_j$,
 \, $i, j =1, \ldots , n$.
 With the usual summation convention, for any $\, \phi \in \Gb$,
\begin{align*}  
\phi^* ({\om}^i_j ) \, = \, {\om}^i_j + \g_{jk}^i (\phi) \,{\t}^k  
\, = \,{\om}^i_j +  \left( {\bf y}^{-1} \cdot
{\phi}^{\prime} (x)^{-1} \cdot \part_{\mu} {\phi}^{\prime} (x) \cdot
{\bf y}\right)^i_j \, dx^\mu ,
\end{align*}
since
\begin{equation*}  
 \g^i_{j \, k} (\phi) (x, {\bf y}) = \left( {\bf y}^{-1} \cdot
{\phi}^{\prime} (x)^{-1} \cdot \part_{\mu} {\phi}^{\prime} (x) \cdot
{\bf y}\right)^i_j \, {\bf y}^{\mu}_k \quad \text{and} \quad \t^k =
 \left({\bf y}^{-1}\right)^k_\ell dx^\ell .
\end{equation*}
Thus, denoting 
\begin{align} \label{Jac}
 \tilde\G _\mu (\phi)(x, {\bf y}) = 
{\bf y}^{-1} \cdot {\phi}^{\prime} (x)^{-1} \cdot \part_{\mu} {\phi}^{\prime} (x) \cdot {\bf y},
\end{align}
one has \,
$ \phi^* (\om_\nb) \, = \, \om_\nb \, + \,  \tilde\G _\mu (\phi) \, dx^\mu$.
Therefore, the simplicial connection is
\begin{align*}  
 \hat{\om}_\nb (\tb ; \phi_0, \ldots , \phi_p)=
 \sum_{r=0}^p t_r \phi_r^* (\om_\nb) 
 =  \om_\nb + \sum_{r=0}^p t_r \,  \tilde\G _k (\phi_r) \, dx^k . 
\end{align*}
Since $\, \phi^*(\Om_\nb ) =0$, the simplicial curvature  \eqref{scurv} 
takes the form
\begin{align*} 
 \hat{\Om}_\nb (\tb ; \phi_0, \ldots , \phi_p) &= 
 \sum_{r=0}^p dt_r \wdg  \phi_r^* (\om_\nb) 
 - \sum_{r=0}^p t_r \,
\phi_r^* (\om_\nb) \wdg \phi_r^* (\om_\nb) \\ \notag
&+ \sum_{r, s=0}^p t_r t_s \,  \phi_r^* (\om_\nb) \wdg  \phi_s^* (\om_\nb) .
 \end{align*}
 Furthermore, being given by invariant polynomials, 
 the Chern cocycles \eqref{chern-forms} are built out of the pull-back of the curvature
form by the cross-section $x \in \Rb^n \mapsto (x, {\bf 1})\in \Rb^n \times\GL_n$.
The latter is given by the matrix-valued form
\begin{align*} 
 \hat{R}&(\tb ; \phi_0, \ldots , \phi_p) \, = 
 \sum_{r=0}^p dt_r \wdg  \G (\phi_r)
 - \sum_{r=0}^p t_r\,
\G(\phi_r)\wdg \G (\phi_r)  \\ 
&+ \sum_{r, s=0}^p t_r t_s \, \G (\phi_r)\wdg \G(\phi_s)  ,  \qquad \text{where} \quad
\G (\phi):= ({\phi}^{\prime})^{-1} \cdot d{\phi}^{\prime} ,
 \end{align*}
with
$\phi^{\prime} = \left(\p_j\phi^i \right)$ denoting the Jacobian matrix of  $\phi \in \Gb$.
 This ensures that the diffeomorphisms $\phi \in \Gb$ appear in 
all the basic cocycles \eqref{basis-forms}
solely through the matrix-valued $1$-forms $\G (\phi) \in \Om^1 (\Rb^n) \ot \Fg\Fl_n$. 
For example, the Chern cocycle  $C_q(\hat\Om_\nb)$  has components 
  \begin{align*}  
&C_q^{(p)}(\hat\Om_\nb)(\phi_0,\ldots, \phi_p) \, = \,   \\
 & = (-1)^p \sum_{1\le i_1 < \ldots < i_q \le n} \sum_{\mu \in S_q} 
 (-1)^\mu \int_{\D^p}R^{i_1}_{\mu(i_1)}\wdg \cdots \wdg R^{i_q}_{\mu(i_q)} 
 (\tb ; \phi_0, \ldots , \phi_p) .
 \end{align*}
In particular, up to a constant factor \, $C_q^{(q)}(\hat\Om_\nb)(\phi_0,\ldots, \phi_q)$
equals
  \begin{align*}  
 \sum_{\s \in S_{q+1}}
  (-1)^\s \, \Tr\left(\G (\phi_{\s(1)}) \wdg \cdots \wdg \G (\phi_{\s(q)})\right) ,
\end{align*}
 where $\s$ runs through the permutations of $\{0, 1, \ldots , q\}$.
 
It is thus seen that every cohomology class in $HP^\bullet (\Hc_n , \GL_n ; \Cb_\d)$
 can be represented by cocycles
$c \in \sum_{q \geq 0} \left(\Qc_n^{\ot^q}\right)^{\GL_n} $ 
whose characteristic image  
 $\chi_{\rm base} (c) \in \sum_{q \geq 0} C^q \left(C_c^{\ify} (\Rb^n ) \rtimes \Gb\right)$ 
  involves only the jet of order $2$. 
 
The above property can be stated more intrinsically,  
in terms of the standard Hopf cyclic complex. Let $\Fc^\d_n$
 denote the subalgebra of $\Hc_n$ generated by the multiplication operators
 $\d_{jk}^i $ of \eqref{gijk} and set
   \begin{align*}
 \Xc_n := \, \Fc^\d_n\, + \,  \sum_{k=1}^n \Fc^\d_n \cdot X_k ;
 \end{align*}
it is a $\GL_n$-invariant subspace of $\Hc_n$, and we let $\dot{\Xc}_n $ be its image 
in $ \Qc_n$.
  
 \begin{cor} \label{1gen}
Every cohomology class in $HP^\bullet (\Hc_n , \GL_n ; \Cb_\d)$
can be represented by cocycles formed of elements in
$\sum_{q \geq 0} \left(\dot\Xc_n^{\ot^q}\right)^{\GL_n} $. 
 \end{cor} 

\begin{proof}
 The horizontal operators appear because of the first summand in the definition
  \eqref{dbo} of the differential $\dbo$,
 which contributes to the  formula \eqref{Phim} as follows:
when applied to monomials
$a = f \, U^*_\phi \in C_c^{\ify} (\Rb^n )$, it brings in the forms
$\, df  = \sum_{k=1}^n X_k (f) \, dx^k$.
 \end{proof}
 
 Explicit representatives for the Hopf cyclic Chern classes can also be
given in the cohomological models of Chevalley-Eilenberg type constructed
in~\cite{MR09, MR11}, by transporting the equivariant Chern classes from
the Bott complex as in \cite[\S 3]{DHC}, 
via the partial inverse of the map $\Theta$ therein defined, only 
this time restricted to $\GL_n$-basic forms. 
 


\begin{thebibliography}{9}
 
\bibitem{Bott*}
 Bott, R., {On characteristic classes in the framework of Gelfand-Fuks cohomology}.
 In  \textit{Colloque ``Analyse et Topologie'' en l'Honneur de Henri Cartan},
  Ast\'erisque {\bf 32-33} (1976), p.  113--139. Soc. Math. France (Paris).
 
  
 \bibitem{BSS}
 Bott, R.,  Shulman, H. and Stasheff, J.,  {On the de Rham theory of certain classifying
 spaces}, {\it Adv. Math.} \textbf{20}  (1976), 43--56.
 
  
 \bibitem{Cext}
Connes, A.,  Cohomologie cyclique et foncteur $Ext^n$,
\textit{C.R. Acad. Sci. Paris}, Ser. I  Math., \textbf{296} (1983),
953-958.
  
\bibitem{book} 
{Connes, A.}, {\bf Noncommutative geometry},
 Academic Press, 1994.

\bibitem{CM98} 
Connes, A. and Moscovici, H., {Hopf algebras, cyclic
cohomology and the transverse index theorem}, \textit{Commun. Math.
Phys.} 198 (1998), 199-246.


\bibitem{CM99} 
Connes, A. and Moscovici, H., {Cyclic cohomology and Hopf
algebras}.
 \textit{Lett. Math. Phys.}  {\bf 48}  (1999),  no. 1, 97--108.

\bibitem{CM01} 
Connes, A. and Moscovici, H., Differentiable cyclic
cohomology and Hopf algebraic structures in transverse geometry,
In \textbf{Essays on Geometry and Related Topics}, pp. 217-256,
\textit{Monographie No. 38 de L'Enseignement Math\'ematique},
Gen\`eve, 2001.


 
 \bibitem{Dupont} 
 Dupont, J. L., {Simplicial de {R}ham cohomology and characteristic classes of
              flat bundles}.
 \textit{Topology}  {\bf 15}  (1976),  233--245.

 
 \bibitem{GF} Gelfand, I. M. and Fuks, D. B., Cohomology of the
Lie algebra of formal vector fields, \textit{Izv. Akad. Nauk SSSR}
\textbf{34} (1970), 322-337.


\bibitem{Godbillon}
Godbillon, C., Cohomologies d'alg\`ebres de Lie de champs de vecteurs formels,
\textit{S\'eminaire N. Bourbaki}, 1972-1973, exp. 421, 69-87.


 \bibitem{HaefDC}
 Haefliger, A., Differentiable cohomology,  In
 \textit{Differential Topology -- Varenna, 1976}, pp. 19-70,
Liguori, Naples, 1979.

  
\bibitem{Kumar} 
Kumar, S. and  Neeb, K-H.,
Extension of algebraic groups,   {\bf Studies in Lie theory}, p. 365--376,
Progr. Math. {\bf 243}, Birkh\"auser Boston, Boston, MA.

\bibitem{MR09} 
Moscovici, H., Rangipour, B.,  Hopf algebras of primitive Lie pseudogroups and Hopf cyclic cohomology. {\it Adv. Math.} \textbf{ 220}  (2009), 706--790.


\bibitem{MR11} 
Moscovici, H., Rangipour, B., Hopf cyclic cohomology
and transverse characteristic classes.  {\it Adv. Math.} \textbf{ 227}  (2011), 654--729.

\bibitem{DHC}
Moscovici, H., Geometric construction of Hopf cyclic characteristic classes.
arXiv:1404.5936,  {\it Adv. Math.}, in press.

\end{thebibliography}
 \end{document}